\documentclass[final,1p,times]{elsarticle}

\usepackage{amssymb}
\usepackage{amsthm}
\usepackage{amscd}
\usepackage{bm}
\usepackage{amsmath}
\usepackage{amsfonts}
\usepackage{graphicx}
\newtheorem{theorem}{Theorem}[section]
\newtheorem{remark}[theorem]{Remark}

\newtheorem{lemma}{Lemma}[section]

\usepackage{mathrsfs}
\usepackage{titletoc}
\usepackage{color}

\newcommand{\ra}{\rightarrow}

\newcommand{\f}{\frac}
\newcommand{\g}{\gamma}

\newcommand{\be}{\begin{equation}}
\renewcommand{\ra}{\rightarrow}
\newcommand{\ee}{\end{equation}}
\newcommand{\bea}{\begin{eqnarray}}
\newcommand{\eea}{\end{eqnarray}}
\newcommand{\bna}{\begin{eqnarray*}}
\newcommand{\ena}{\end{eqnarray*}}

\renewcommand{\le}{\left}
\newcommand{\ri}{\right}

\usepackage{geometry}
\journal{***}

\begin{document}

\begin{frontmatter}

\title{Global existence and multiplicity of solutions for logarithmic Schr\"{o}dinger equations on graphs}

\author[qnu]{Mengqiu Shao\corref{Shao}}
\ead{shaomamath@163.com}
\address[qnu]{School of Mathematical Sciences, Qufu Normal University, Shandong, 273165, China}

\cortext[Shao]{Corresponding author.}

\begin{abstract}
We consider the following logarithmic Schr\"{o}dinger equation
  $$
  -\Delta u+h(x)u=u\log u^{2}
  $$
on a locally finite graph $G=(V,E)$, where $\Delta$ is a discrete Laplacian operator on the graph, $h$ is the potential function. Different from the classical methods in Euclidean space, we obtain the existence of global solutions to the equation by using the variational method from local to global, which is inspired by the works of Lin and Yang in \cite{LinYang}. In addition, when the potential function $h$ is sign-changing, we prove that the equation admits infinitely many solutions with high energy by using the symmetric mountain pass theorem. We extend the classical results in Euclidean space to discrete graphs.
\end{abstract}


\end{frontmatter}

\section{Introduction}

In recent years, the partial differential equation on graphs has become a hot topic in the field of discrete geometry and nonlinear analysis.  For example, Lin and Wu \cite{Lin1} studied the existence and nonexistence of global solutions for a semilinear heat equation on graphs. Hua, Keller, Schwarz and Wirth \cite{HuaKeller} investigated the eigenvalue growth of infinite graphs with
discrete spectrum. Cushing, Kamtue, Liu and Peyerimhoff \cite{DSLN} reformulated the Bakry-\'{E}mery curvature on a weighted graph. Moreover, the existence and asymptotic behavior of solutions to elliptic equations have been widely studied on graphs, see \cite{HanShao, HanShaoZhao} for more details. In particular, Lin and Yang \cite{LinYang} considered various equations on locally finite graphs and established a variational method from local to global on graphs. More recently, Chang, Wang and Yan \cite{ChangYan1} first investigated the existence of ground state solutions of a Schr\"{o}dinger equation with logarithmic nonlinearity on locally finite graphs. Later, Chang et al. \cite{ChangYan2} considered the existence of least energy sign-changing solutions of a logarithmic Schr\"{o}dinger equation on locally finite graphs. The above works arouse our great research interest. In this paper, we are focus on the following logarithmic Schr\"{o}dinger equation
\begin{align}\label{log}
-\Delta u+h(x)u=u\log u^{2}
\end{align}
on a locally finite graph $G=(V,E)$, where $\Delta$ is the discrete Laplacian on graphs and $h$ is the potential function satisfying certain conditions. On the Euclidean space, equation \eqref{log} is closely related to the time-dependent logarithmic Schr\"{o}dinger equation
\begin{equation*}
  i\frac{\partial u}{\partial t}-\Delta u+h(x)u-u\log u^2=0\ \ \hbox{in}\ \ \mathbb{R}^{N}\times\mathbb{R}_{+},
\end{equation*}
which has important applications in quantum mechanics, quantum optics, nuclear physics, transport and diffusion
phenomena, open quantum systems, effective quantum gravity, theory of superfluidity, and Bose-Einstein condensation (see \cite{IJ,TC,WangZhang,KG} for more details).

When $V=\mathbb{R}^N$, the equation \eqref{log} is formally associated with the energy functional $J: W^{1,2}(\mathbb{R}^{N})\rightarrow \mathbb{R}\cup\{+\infty\}$ defined by
\begin{equation*}
 J(u)=\frac{1}{2}\int_{\mathbb{R}^{N}}|\nabla u|^2dx+\frac{1}{2}\int_{\mathbb{R}^{N}}(h(x)+1)u ^2dx-\frac{1}{2}\int_{\mathbb{R}^{N}}u^{2}\log u^{2}dx.
\end{equation*}
However,  in general, this functional fails to be finite and $C^1$ smooth on $W^{1,2}(\mathbb{R}^{N})$ because there exists $u\in W^{1,2}(\mathbb{R}^{N})$ such that $\int_{\mathbb{R}^{N}}u^{2}\log u^{2}dx=-\infty$. Thus the classical critical point theory cannot be applied for $J$. To overcome this obstacle, several approaches have been developed on the Euclidean space. For instance, Cazenave \cite{TC} worked in an Orlicz space endowed with Luxemburg type norm in order to make the associated functional well defined and  $C^1$ smooth. By applying non-smooth critical point theory for lower semi-continuous functionals, Squassina and Szulkin \cite{MA} proved the existence of multiple solutions to a logarithmic Schr\"{o}dinger equation with periodic potential. By using the penalization technique, Tanaka and Zhang \cite{KC} obtained infinitely many multi-bump geometrically distinct solutions for a logarithmic Schr\"{o}dinger equation.  For a more non-smooth variational framework to the logarithmic Schr\"{o}dinger equation,  we refer the reader to\cite{AH,Alves1,Alves2,FengTangZhang,PJ,Shuai}.

It is worth pointing out that the general form of the equation (\ref{log}) is the following Schr\"{o}dinger equation
\begin{equation}\label{sch}
  -\Delta u+h(x)u=f(x,u),\ \ x\in V.
\end{equation}
On a lattice graphs, when the potential function $h$ is periodic or bounded, Hua and Xu \cite{HuaXu} proved the existence of ground state solutions for the equation \eqref{sch} by using the Nehari manifold method.
When $h(x)=\lambda a(x)+1$ and $f(x,u)=|u|^{p-1}u$, $p\geq 2$, Zhang and Zhao \cite{ZhangZhao} proved the existence and convergence of ground state solutions to the equation \eqref{sch} on locally finite graphs.
Under the following Ambrosetti-Rabinowitz (or (A-R) for short) condition
\vskip4pt
\noindent $(f_{1})$ \emph{there exists a constant $\alpha>2$ such that for all $x\in V$ and $s\in \mathbb{R}\setminus\{0\}$,}
$$0<\alpha F(x,s)=\alpha\int^{s}_{0}f(x,t)dt\leq sf(x,s)$$
\vskip4pt
\noindent and some growth conditions, Grigor'yan, Lin and Yang \cite{GLY3} obtained that the equation (\ref{sch}) admits a positive solution on graphs, where the (A-R) condition $(f_{1})$ plays an important role in proving the boundedness of $(PS)$ sequence.

However, the logarithmic nonlinearity $u\log u^2$  in the equation (\ref{log}) not satisfies the  (A-R) condition,  which brings some difficulties when we use the variational method to study such equations on graphs. On the other hand, similarly to the case in Euclidean space, the corresponding energy functional of the equation \eqref{log} is not well defined on $W^{1,2}(V)$. Particularly, Chang, Wang and Yang \cite{ChangYan1} gave the following example
$$u(x)=\le\{\begin{array}{lll}
(|x|\log|x|)^{-1},& |x|\geq3,\\[1.5ex]
0,& |x|\leq2
\end{array}\ri.$$
and verified $u\in W^{1,2}(V)$ but $\int_{V}u^{2}\log u^{2}d\mu=-\infty$, where $x_0\in V$ is some fixed point and $|x|:=\rho(x,x_0)$ denotes the distance between $x$ and $x_0$ on graphs. To deal with this limitation, they restricted $u^2\log u^2\in L^1(V)$ and applied the direction derivative and the Nehari manifold method to obtain the existence of ground state solutions of the equation \eqref{log} in \cite{ChangYan1}.

Inspired by the above works,  this paper  is devoted to study the existence and multiplicity of solutions to the logarithmic Schr\"{o}dinger equation \eqref{log} on the locally finite graph $G=(V,E)$. Specifically, our first aim is to obtain the existence of global solutions to the equation \eqref{log} by using the variational method from local to global, which is different from the classical methods in Euclidean space and the analysis techniques in \cite{ChangYan1}. Roughly speaking, we first get sequences of local solutions $u_k$ for equation \eqref{log} on $W^{1,2}(B_k)$ and derive uniform estimates for those local solution sequences. Then we obtain the global solutions by extracting convergent sequence of local solutions. We have to point out that one of the difficulties is to prove the nontrivial property of the global solutions. On the other hand, under the assumption that the potential function $h$ allows changing sign, we also prove that the equation \eqref{log} admits infinitely many solutions with high energy by using the symmetric mountain pass theorem. To the best of our knowledge, there seem no results concerned with infinitely many solutions to such equations on graphs.

To describe our main results in details, we first give some notations and assumptions.
Throughout this paper, we always assume that $G=(V,E)$ is a connected locally finite graph and its measure $\mu:V\rightarrow \mathbb{R}^{+}$ is a finite positive function satisfying
\begin{equation}\label{mu}
  \mu(x)\geq\mu_{min}>0\ \  \hbox{for\ \  all}\ \  x\in V.
\end{equation}
For an edge $xy\in E$, the weight $\omega_{xy}$ is positive and symmetric, namely
\begin{equation}\label{omega}
 \omega_{xy}=\omega_{yx}>0.
\end{equation}
The distance $\rho(x,y)$ of two vertices $x,y\in V$ is defined by the minimal number
of edges which connect these two vertices. For any function $u:V\rightarrow \mathbb{R}$ and $x\in V$, the Laplacian  of $u$ at $x$ is defined by
$$\Delta u(x):=\frac{1}{\mu(x)}\underset{y\sim x}\sum \omega_{xy}(u(y)-u(x)).$$
For any fixed $x\in V$, the gradient of $u$ at $x$ is given by
$$
\nabla u(x)=\left(\sqrt{\frac{\omega_{xy_1}}{2\mu(x)}}(u(y_1)-u(x)),\cdots,\sqrt{\f{\omega_{xy_{\ell_x}}}{2\mu(x)}}(u(y_{\ell_x})-u(x))\right),
$$
where $\{y_1,\cdots,y_{\ell_x}\}$ is the set of all neighbors of $x$ for some positive integer $\ell_x$. Then $\nabla u(x)$ can be seen as a vector in $\mathbb{R}^{\ell_x}$.
 $\nabla u\nabla v $ at $x\in V$ is denoted by
$$\nabla u\nabla v (x)=\frac{1}{2\mu(x)}\underset{y\sim x}\sum\omega_{xy}(u(y)-u(x))(v(y)-v(x)).$$
The length of the gradient for $u$ is denoted by
$$|\nabla u(x)|=\left(\frac{1}{2\mu(x)}\underset{y\sim x}\sum\omega_{xy}(u(y)-u(x))^{2}\right)^{\frac{1}{2}}.$$
For any function $u: V\rightarrow\mathbb{R}$, an integral of $u$ over $V$ is defined by
$$\int_{V}ud\mu=\underset{x\in V}\sum\mu (x)u(x).$$

For any $1\leq s<\infty$, we denote by
$$L^{s}(V):=\left\{u:V\rightarrow\mathbb{R}:\int_{V}|u|^{s}d\mu<\infty\right\}$$
the set of integrable functions on $V$ with the respect to the measure $\mu$. For $s=\infty$, let
\begin{equation*}
 L^{\infty}(V):=\left\{u:V\rightarrow\mathbb{R}:\underset{x\in V}\sup|u(x)|<\infty\right\}.
\end{equation*}
Usually, we use $\|\cdot\|_{s}$ to denote the $L^s(V)$ norm. Define
\begin{equation*}
W^{1,2}(V):=\left\{u:V\rightarrow\mathbb{R}: \int_{V}(|\nabla u|^{2}+u^{2})d\mu<+\infty\right\}
\end{equation*}
with the norm
$$\|u\|_{W^{1,2}(V)}:=\left(\int_{V}(|\nabla u|^{2}+u^{2})d\mu\right)^{\frac{1}{2}}.$$
Clearly, $W^{1,2}(V)$ is a Hilbert space with the inner product
$$(u,v)_{W^{1,2}(V)}=\int_{V}(\nabla u\nabla v+uv)d\mu,\ \ \forall u,v\in W^{1,2}(V).$$
Let $h(x)\geq h_0> 0$ for all $x\in V$, we define a subspace
of $W^{1,2}(V)$, which is also a Hilbert space, namely
\begin{equation}\label{H}
  \mathcal{H}:=\{u\in W^{1,2}(V):\int_{V} h(x)u^{2}d\mu<+\infty\}
\end{equation}
with the norm
\begin{equation*}
  \|u\|_{\mathcal{H}}:=\left(\int_{V}(|\nabla u|^{2}+ h(x)u^{2})d\mu\right)^{\frac{1}{2}}.
\end{equation*}

Our first result is as follows.

\begin{theorem}\label{existence1}
Let $G=(V,E)$ be a connected locally finite graph satisfying \eqref{mu} and \eqref{omega}. Assume that $h$ satisfies the following conditions:
\begin{enumerate}
  \item  [($h_{1}$)] there exists a constant $h_0>0$ such that $h(x)\geq h_0$ for all $x\in V$;
  \item  [($h_{2}$)] $h^{-1}\in L^{1}(V)$.
\end{enumerate}
Then the equation \eqref{log} possesses a nontrivial solution.
\end{theorem}

\begin{remark}
\emph{The assumptions $(h_1)$ and $(h_2)$ were originally introduced on graphs by Grigor'yan, Lin and Yang in  \cite{GLY3} which is used to guarantee the compactness of the embedding $\mathcal{H}\hookrightarrow L^s(V)$, $s\in [1,+\infty]$. Especially, this embedding theorem plays an important role in proving the nontrivial property of the global solutions.}
\end{remark}

Moreover, we consider the existence of infinitely many solutions to the equation \eqref{log} when $h$ satisfies the following  assumptions.
\begin{enumerate}
\item  [($h'_{1}$)] \emph{$\underset{x\in V}\inf h(x)\geq h_{1}$ for some constant $h_{1}\in(-1,0)$.}
  \item  [($h'_{2}$)]\emph{ There exists $\alpha>0$ such that $h^{-1}\in L^{1}(V\setminus V_{\alpha})$, where $V_{\alpha}=\{x\in V:h(x)\leq \alpha\}$ and the volume of $V_{\alpha}$ is finite, i.e.}
      \begin{equation*}
        Vol(V_{\alpha})=\underset{x\in V_{\alpha}}\sum\mu(x)<\infty.
      \end{equation*}
\end{enumerate}
Note that the equation \eqref{log} is formally associated with the energy functional $\mathcal{J}: W^{1,2}(V)\rightarrow \mathbb{R}\cup\{+\infty\}$ defined by
\begin{equation}\label{functionalJ}
\mathcal{J}(u)=\frac{1}{2}\int_{V}|\nabla u|^2d\mu+\frac{1}{2}\int_{V}(h(x)+1)u^2d\mu-\frac{1}{2}\int_{V}u^{2}\log u^{2}d\mu.
\end{equation}
To study the  multiplicity of solutions to the equation \eqref{log}, we consider a function space
\begin{equation}\label{fW}
 \mathcal{W}:=\{u\in W^{1,2}(V): \int_{V}(h(x)+1)u^2d\mu<+\infty\}
\end{equation}
with the norm
$$\|u\|_{\mathcal{W}}:= \left(\int_{V}[|\nabla u|^2+(h(x)+1)u^2]d\mu\right)^{\frac{1}{2}},$$
where $h$ satisfies $(h'_{1})$ and $(h'_{2})$.
It follows from Lemma~\ref{smooth} in Section 3 that $\mathcal{J}\in C^{1}(\mathcal{W},\mathbb{R})$.

\vskip4pt
Our second result is as follows.

\begin{theorem}\label{existence2}
Let $G=(V,E)$ be a connected locally finite graph satisfying \eqref{mu} and \eqref{omega}. Suppose that $(h'_{1})$ and $(h'_{2})$  hold. Then the equation \eqref{log} admits a sequence of solutions $\{u_n\}$ with $\mathcal{J}(u_n)\rightarrow+\infty$ as $n\rightarrow\infty$.
\end{theorem}

\begin{remark}
\emph{It should be pointed out that the hypotheses $(h'_{1})$ and $(h'_{2})$ were proposed by Chang, Wang and Yang in \cite{ChangYan1}, where they got the existence of ground state solutions to the equation \eqref{log} on graphs. Different from their results in \cite{ChangYan1}, by using the symmetric mountain pass theorem,  we obtain infinitely many solutions with high energy of the equation \eqref{log}. On the other hand, we extend the classical results in Euclidean space to discrete graphs.}
\end{remark}

This article is organized as follows.  In Sect. 2, we prove the existence of the nontrivial solutions of equation \eqref{log} from local to global. In Sect. 3, we prove Theorem~\ref{existence2} by using the symmetric mountain pass theorem.

\section{Existence of global solutions }
In this section, we prove Theorem \ref{existence1} by using variational method from local to global.  Firstly, we present some preliminaries and basic functional settings. Let $\Omega$ be a connected finite  subset of $V$.
The boundary of $\Omega$ in $V$ is defined by
\begin{equation}\label{boundary}
  \partial\Omega:=\{y\notin\Omega: \exists x\in\Omega\ \ \hbox{such that} \ \ xy\in E\}.
\end{equation}
The space $\mathcal{H}(\Omega):=W^{1,2}_{0}(\Omega)$  endowed with the inner product
$$(u,v)_{\mathcal{H}(\Omega)}=\left(\int_{\overline{\Omega}}|\nabla u|^{2}d\mu+\int_{\Omega}h(x)u^{2}d\mu\right)^{\frac{1}{2}}$$
is a Hilbert space, where $\overline{\Omega}=\Omega\cup\partial \Omega$.
Fixed some point $O\in V$. For any $x\in V$, $\rho(x)=\rho(x,O)$ denotes the distance between $x$ and $O$. For any integer $k\geq 1$, we denote a ball centered at $O$ with radius $k$ by
$$B_{k}=B_{k}(O)=\{x\in V:\rho(x)<k\}.$$
The boundary of $B_{k}$ written as
$$\partial B_{k}=\{x\in V:\rho (x)=k\}.$$
It is easy to see that if $\Omega=B_k$, then $\partial\Omega$ defined in \eqref{boundary} is equal to $\partial B_{k}$.

Now we consider the following locally equation
\begin{align}\label{locally}
\begin{cases}
-\Delta u+h(x)u=u\log u^{2}, &\text{in}\ \ B_{k},\\
u=0, &\text{on}\ \ \partial B_{k},
\end{cases}
\end{align}
where $h$ satisfies $(h_1)$ and $(h_2)$.
It is suitable to study \eqref{locally} in the space $\mathcal{H}(B_{k})$ under the norm
$$\|u\|_{\mathcal{H}(B_{k})}=\left(\int_{\overline{B}_{k}}|\nabla u|^{2}d\mu+\int_{B_{k}}h(x)u^{2}d\mu\right)^{\frac{1}{2}}.$$
The functional related to \eqref{locally} is
\begin{equation*}
\mathcal{I}_{k}(u)=\frac{1}{2}\|u\|^2_{\mathcal{H}(B_{k})}+\frac{1}{2}\int_{B_{k}}u^2d\mu-\frac{1}{2}\int_{B_{k}}u^{2}\log u^{2}d\mu.
\end{equation*}
Note that for any $0<\varepsilon<1$, there exists $C_{\varepsilon}>0$ such that
\begin{equation}\label{inequality1}
  |t^2\log t^2|\leq C_{\varepsilon}(|t|^{2-\varepsilon}+|t|^{2+\varepsilon}),\ \ \forall t\in \mathbb{R}\backslash\{0\}.
\end{equation}
By \eqref{inequality1} and a standard argument, we have $\mathcal{I}_{k}\in C^{1}(\mathcal{H}(B_{k}),\mathbb{R})$, and
\begin{equation*}
\langle \mathcal{I}'_{k}(u),v\rangle=\int_{\overline{B}_{k}}\nabla u\nabla vd\mu+\int_{B_{k}}h(x)uvd\mu-\int_{B_{k}}u^{2}\log u^{2}d\mu.
\end{equation*}

Next, we present a result about the compactness of the Sobolev space $\mathcal{H}$ defined by \eqref{H}.
\begin{lemma}\label{e1}
Let $G=(V,E)$ be a connected locally finite graph satisfying \eqref{mu} and \eqref{omega}. Suppose that $(h_{1})$ and $(h_{2})$ hold.  Then $\mathcal{H}$ is weakly pre-compact and $\mathcal{H}$ is compactly embedded into $L^{q}(V)$ for any $q\in [1,+\infty]$.
\end{lemma}

\begin{proof}
The proof of this lemma is similar to Lemma 2.6 in \cite{ZhangZhao}, so we omit it.
\end{proof}

The following lemma gives the compactness of the space $\mathcal{H}(\Omega)$. In addition, we prove an embedding inequality where the embedding constant is independent of $\Omega$. For brevity, we denote the $L^{q}$ norms on $\Omega$ by  $\|\cdot\|_{q,\Omega}$.

\begin{lemma}\label{Sobolevembedding}
 Let $G=(V,E)$ be a connected locally finite graph satisfying \eqref{mu} and \eqref{omega}, $\Omega$ be a connected finite subset of $V$. Assume that $(h_{1})$ hold. Then $\mathcal{H}(\Omega)$ is compactly embedded into $L^{q}(\Omega)$ for any $q\in [2,+\infty]$. In particular, there exists a constant $C$ depending only on $h_0$, $\mu_{\min}$ and $q$ such that for any $u\in \mathcal{H}(\Omega)$,
\begin{equation}\label{Sin}
\|u\|_{q,\Omega}\leq C\|u\|_{\mathcal{H}(\Omega)}.
\end{equation}
Moreover, $\mathcal{H}(\Omega)$ is  pre-compact.
\end{lemma}

\begin{proof} The proof of this lemma is inspired by  Theorem~7 in \cite{GLY3}. Since $\Omega$ is a connected finite subset of $V$, $\mathcal{H}(\Omega)$ is a finite dimensional space. Hence, $\mathcal{H}(\Omega)$ is  pre-compact. We are left to prove (\ref{Sin}). Indeed, for any $u\in \mathcal{H}(\Omega)$ and vertex $x_{0}\in \Omega$, we have
\begin{align*}
\|u\|^{2}_{\mathcal{H}(\Omega)}&=\int_{\overline{\Omega}}|\nabla u|^{2}d\mu+\int_{\Omega}h(x)u^{2}d\mu\\
&\geq \int_{\Omega}h(x)u^{2}d\mu\\
&=\underset{x\in \Omega}\sum h(x)u(x)^{2}\mu(x)\\
&\geq h_0\mu_{\min}u(x_{0})^{2},
\end{align*}
which gives
\begin{equation*}
u(x_{0})\leq\left(\frac{1}{h_0\mu_{\min}}\right)^{\frac{1}{2}}\|u\|_{\mathcal{H}(\Omega)}.
\end{equation*}
Then
\begin{equation*}
 \|u\|_{\infty,\Omega}\leq\left(\frac{1}{h_0\mu_{\min}}\right)^{\frac{1}{2}}\|u\|_{\mathcal{H}(\Omega)},
\end{equation*}
and thus for any $2\leq q<+\infty$,
\begin{align*}
   \int_{\Omega}|u|^{q}d\mu&=\int_{\Omega}|u|^{2}|u|^{q-2}d\mu\\
 &\leq\left(\frac{1}{h_0\mu_{\min}}\right)^{\frac{q-2}{2}}\|u\|^{q-2}_{\mathcal{H}(\Omega)}\int_{\Omega}u^2d\mu\\
 &\leq(h_0)^{\frac{-q}{2}}(\mu_{\min})^{\frac{2-q}{2}}\|u\|^{q-2}_{\mathcal{H}(\Omega)}\int_{\Omega}h(x)u^2d\mu\\
 &\leq(h_0)^{\frac{-q}{2}}(\mu_{\min})^{\frac{2-q}{2}}\|u\|^{q}_{\mathcal{H}(\Omega)}.
\end{align*}
Therefore, for any $2\leq q\leq+\infty$,
\begin{equation*}
 \|u\|_{q,\Omega}\leq C\|u\|_{H(\Omega)},
\end{equation*}
where $C$ depends only on $h_0$, $\mu_{\min}$ and $q$.
\end{proof}

Now, we list two lemmas about Green's formulas on graphs, which are fundamental when we use methods in calculus of variations.

\begin{lemma}[{\cite[Lemma~2.1]{ZhangZhao}}] \label{a}
 Let $G=(V,E)$ be a connected locally finite graph satisfying \eqref{mu} and \eqref{omega}. Then for any $v\in C_{c}(V)$, we have
\begin{equation*}\label{partw12v}
\int_{V}\nabla u\nabla vd\mu=-\int_{V}(\Delta u)vd\mu,
\end{equation*}
where $C_c(V)$ is the set of all functions with finite supports in $V$.
\end{lemma}

\begin{lemma}[{\cite[Lemma~2.2]{ZhangZhao}}]\label{b}
 Let $G=(V,E)$ be a connected locally finite graph satisfying \eqref{mu} and \eqref{omega}, $\Omega$ be a connected finite subset of $V$ and $u\in W_0^{1,2}(\Omega)$. Then for any $v\in C_{c}(\Omega)$, we have
\begin{equation*}\label{partw12omega}
\int_{\overline{\Omega}}\nabla u\nabla vd\mu=-\int_{\Omega}(\Delta u)vd\mu,
\end{equation*}
where $C_{c}(\Omega)$ denotes the set of all functions $u: \Omega\rightarrow\mathbb{R}$ satisfying
$\emph{\hbox{supp}}\  u\subset\Omega$ and $u=0$ on $\partial\Omega$.
\end{lemma}

Next,  we verify that the functional $\mathcal{I}_k$ satisfies the mountain pass geometry
in the following two lemmas.

\begin{lemma}\label{mp1}
   There exist constants $\delta>0$ and $r>0$ satisfying
  $\mathcal{I}_k(u)\geq \delta$ for all $u$ with $\|u\|_{\mathcal{H}(B_{k})}=r$.
  \end{lemma}
\begin{proof}
It follows from \eqref{inequality1} and Lemma~\ref{Sobolevembedding} that
\begin{align*}
\mathcal{I}_{k}(u)&=\frac{1}{2}\|u\|^2_{\mathcal{H}(B_{k})}+\frac{1}{2}\int_{B_{k}}u^2d\mu-\frac{1}{2}\int_{B_{k}}u^{2}\log u^{2}d\mu\\
&\geq\frac{1}{2}\|u\|^2_{\mathcal{H}(B_{k})}-\frac{1}{2}\int_{\{x\in B_{k}:u(x)\geq 1\}}u^2\log u^{2}d\mu\\
&\geq\frac{1}{2}\|u\|^2_{\mathcal{H}(B_{k})}-\frac{1}{2}\int_{B_{k}}C_{\varepsilon}|u|^{2+\varepsilon}d\mu\\
&\geq\frac{1}{2}\|u\|^2_{\mathcal{H}(B_{k})}-\frac{C_{\varepsilon}}{2}\|u\|^{2+\varepsilon}_{2+\varepsilon,B_{k}}\\
&\geq\frac{1}{2}\|u\|^2_{\mathcal{H}(B_{k})}-
\frac{C_{\varepsilon}}{2}C\|u\|^{2+\varepsilon}_{\mathcal{H}(B_{k})}\\
&=\frac{1}{2}\|u\|^2_{\mathcal{H}(B_{k})}\left(1-C_{\varepsilon}C\|u\|^{\varepsilon}_{\mathcal{H}(B_{k})}\right),
\end{align*}
where $C$ depends on $h_0$, $\mu_{\min}$ and $\varepsilon$.
Setting $r=\left(\frac{1}{2C_{\varepsilon}C}\right)^{\frac{1}{\varepsilon}}$, we have
  $\mathcal{I}_k(u)\geq \frac{1}{4}r^2$ for all $u$ with $\|u\|_{\mathcal{H}(B_{k})}=r$. This gives the desired result.
\end{proof}

\begin{lemma}\label{mp2}
There exists some $u\in \mathcal{H}(B_{k})$ with $u(x)\geq 0$ for all $x\in B_{k}$ such that $\mathcal{I}_k(tu)\ra-\infty$ as $t\ra+\infty$.
\end{lemma}
\begin{proof}
For any fixed point $x_0\in B_{k}$, we set
$$u(x)=\le\{\begin{array}{lll}
1&{\rm when}& x=x_0\\[1.5ex]
0&{\rm when}& x\not=x_0.
\end{array}\ri.$$
Then we have
\begin{equation*}
\mathcal{I}_{k}(tu)=\f{t^2}{2}\left(\sum_{y\sim x_0}\mu(y)|\nabla u|^2(y)+\mu(x_0)|\nabla u|^2(x_0)\right)+\f{t^2}{2}\mu(x_0)(h(x_0)+1)-\frac{t^2}{2}\log t^{2}\mu(x_{0})\ra-\infty
    \end{equation*}
as $t\ra+\infty$.
\end{proof}

Next, we prove that $\mathcal{I}_{k}$ satisfies the $(PS)_{c}$ condition for any $c\in \mathbb{R}$.

\begin{lemma}\label{ps-cond} For any $c\in \mathbb{R}$, if $\{u_n\}\subset \mathcal{H}(B_{k})$ satisfies $\mathcal{I}_{k}(u_n)\ra c$ and $\mathcal{I}_{k}'(u_n)\ra 0$, then there exists some $u\in\mathcal{H}(B_{k})$ such that, up to a subsequence, $u_n\ra u$ in $\mathcal{H}(B_{k})$.
\end{lemma}

\begin{proof}
Note that $\{u_n\}\subset \mathcal{H}(B_{k})$ satisfies
\begin{equation}\label{I}
   \mathcal{I}_{k}(u_n)\ra c\ \ \hbox{and}\ \ \mathcal{I}_{k}'(u_n)\ra 0\ \ \hbox{as}\ \ n\ra\infty.
\end{equation}
We claim that $\{u_{n}\}$ is bounded in $\mathcal{H}(B_{k})$. In fact, by \eqref{I} we have
\begin{equation}\label{bdd1}
 \underset{n\rightarrow\infty}\lim \mathcal{I}_{k}(u_n)=\underset{n\rightarrow\infty}\lim \left(\mathcal{I}_{k}(u_n)-\frac{1}{2}\langle\mathcal{I}'_{k}(u_n),u_{n}\rangle\right)
=\frac{1}{2}\underset{n\rightarrow\infty}\lim \int_{B_{k}}u^2_{n}d\mu
=\frac{1}{2}\underset{n\rightarrow\infty}\lim \|u_{n}\|^2_{2,B_{k}}
=c.
\end{equation}
On the other hand, by the H\"{o}lder's inequality, for any $\varepsilon\in (0,1)$, there exist $C_{\varepsilon},C_{1}>0$ such that

\begin{align}\label{bdd2}
\int_{B_{k}}u^2_n\log u_n^2d\mu&\leq\int_{\{x\in B_{k}:u_n(x)\geq 1\}}u^2_n\log u_n^2d\mu\nonumber\\
&\leq C_{\varepsilon}\int_{B_{k}}|u_{n}|^{2+\varepsilon}d\mu\nonumber\\
&\leq C_{\varepsilon}\left(\int_{B_{k}}|u_{n}|^{2}d\mu\right)^{\frac{1}{2}}
\left(\int_{B_{k}}|u_{n}|^{2(1+\varepsilon)}d\mu\right)^{\frac{1}{2}}\nonumber\\
&=C_{\varepsilon}\|u_{n}\|_{2,B_{k}}\|u_{n}\|^{1+\varepsilon}_{2(1+\varepsilon),B_k}\nonumber\\
&\leq C_{\varepsilon}C_{1}\|u_{n}\|_{2,B_{k}}\|u_{n}\|^{1+\varepsilon}_{\mathcal{H}(B_{k})}\nonumber\\
&\leq\frac{1+\varepsilon}{2}\|u_{n}\|^{2}_{\mathcal{H}(B_{k})}+
\frac{1-\varepsilon}{2}C_{\varepsilon}C_{1}\|u_{n}\|^{\frac{2}{1-\varepsilon}}_{2,B_{k}},
\end{align}
where $C_1$ depends on $h_0$, $\mu_{\min}$, $\varepsilon$ and the last inequality is derived from the Young inequality. Since $\langle\mathcal{I}'_{k}(u_{n}),u_{n}\rangle=0$, combining \eqref{bdd1} and \eqref{bdd2}, we obtain
\begin{equation}\label{bdd3}
  \|u_{n}\|^2_{\mathcal{H}(B_{k})} = \int_{B_{k}}u^2_{n}\log u^2_{n}d\mu
   \leq\frac{1+\varepsilon}{2}\|u_{n}\|^{2}_{\mathcal{H}(B_{k})}+
\frac{1-\varepsilon}{2}C_{\varepsilon}C_{1}\|u_{n}\|^{\frac{2}{1-\varepsilon}}_{2,B_{k}},
\end{equation}
which implies that $\{u_n\}$ is bounded in $\mathcal{H}(B_{k})$, and thus there exists a constant $M$  independent of $k$ such that $\|u_n\|_{\infty,B_k}\leq M(c,\varepsilon,\mu_{\min},h_0)$. By Lemma~\ref{Sobolevembedding}, up to a subsequence, there exists some $u\in \mathcal{H}(B_{k})$ such that $u_{n}\rightarrow u$ in $\mathcal{H}(B_{k})$ as $n\rightarrow\infty$.
\end{proof}

Now, we will prove that the local equation \eqref{locally} admits a nontrivial solution.

\begin{lemma}
Let $G=(V,E)$ be a connected locally finite graph satisfying \eqref{mu} and \eqref{omega}. Suppose that $(h_{1})$ hold. Then equation \eqref{locally} has a nontrivial solution.
\end{lemma}
\begin{proof}
By Lemma~\ref{mp1}, Lemma~\ref{mp2} and Lemma~\ref{ps-cond}, $\mathcal{I}_{k}$ satisfies all the hypotheses of
the mountain-pass theorem. Using the mountain-pass theorem due to Ambrosetti-Rabinowitz \cite{MW}, we conclude that $$c_k=\min_{\gamma\in\Gamma}\max_{u\in\gamma}\mathcal{I}_k(u)$$
is the critical level of $\mathcal{I}_k$, where
$$\Gamma=\le\{\g\in C([0,1],\mathcal{H}(B_{k})): \g(0)=0, \g(1)=e\ri\}.$$
In particular, there exists some $u_k\in \mathcal{H}(B_{k})$ such that $\mathcal{I}_k(u_k)=c_k\geq \delta>0$ and $u_{k}\not\equiv 0$ satisfies the
Euler-Lagrange equation
\begin{align}\label{locally1}
\begin{cases}
-\Delta u_k+h(x)u_k=u^2_k\log u_k^{2}, &\text{in}\ \ B_{k},\\
u_k=0, &\text{on}\ \ \partial B_{k}.
\end{cases}
\end{align}
\end{proof}

\vskip4pt

At the end of this section, we prove the existence of global solutions to the equation \eqref{log}.

\vskip8pt
\noindent  \emph{The proof of Theorem \ref{existence1}}.~~We will divide the proof into two steps as follows.

\vskip4pt
\textbf{(i) For any finite set $A\subset V$, the nontrivial solution $u_{k}$ of \eqref{locally} is uniformly bounded in $A$.}
\vskip4pt

Let $u_{k}$ be a nontrivial solution of \eqref{locally}. Similar to the proof of \eqref{bdd3} in Lemma \ref{ps-cond}, we have
\begin{equation}\label{bdd}
  \|u_{k}\|^2_{\mathcal{H}(B_{k})}=\int_{\overline{B}_k}|\nabla u_{k}|^2d\mu+\int_{B_k}h(x)u^2_{k}d\mu\leq C
\end{equation}
for some constant $C$ independent of $k$. Note that for any finite set $A\subset V$, there exists $B_{k}$ such that  $A\subset B_k$ for sufficiently large $k$. Since
\begin{align*}
  \|u_{k}\|^2_{\mathcal{H}(B_{k})}&=\int_{\overline{B}_k}|\nabla u_{k}|^2d\mu+\int_{B_k}h(x)u^2_{k}d\mu\\
  &\geq\int_{B_{k}}h(x)u^2_{k}d\mu\\
  &\geq h_0\int_{A}u^2_{k}d\mu\\
  &\geq h_0\mu_{\min}u^2_k(x_{0})
\end{align*}
for any fixed $x_{0}\in A$, which ensures that
$$\|u_k\|_{\infty, A}\leq\left(\frac{1}{h_{0}\mu_{\min}}\right)
^{\frac{1}{2}} \|u_{k}\|_{\mathcal{H}(B_{k})}\leq\sqrt{\frac{C}{h_{0}\mu_{\min}}}.$$
Thus $u_{k}$ is uniformly bounded in $A$.

\vskip4pt
\textbf{(ii) The equation \eqref{log} admits a nontrivial solution in $\mathcal{H}$. }
\vskip4pt

Note that $u_k$ is naturally viewed as a sequence of functions defined on $V$, say $u_k\equiv0$
on $V\backslash B_k$. Then by the step (i), there would exist a subsequence of $\{u_k\}$, which is still denoted by $\{u_k\}$, and a function $u^{*}$ such that $u_k$ converges to $u^{*}$ locally uniformly in $V$, i.e. for any
fixed positive integer $l$,
\begin{equation*}
  \underset{k\rightarrow\infty}\lim u_{k}(x)=u^{*},\ \ \forall x\in B_{l}.
\end{equation*}
Next we prove that $u^{*}\in \mathcal{H}$. Since $u_k\equiv0$ on $V\backslash B_k$, by \eqref{bdd} we have
\begin{align*}
\|u_{k}\|^2_{\mathcal{H}}&= \int_{V}(|\nabla u_{k}|^2+h(x)u^2_{k})d\mu\\
&=\underset{x\in V}\sum|\nabla u_k|^2(x)\mu(x)+\underset{x\in B_k}\sum h(x)u^2_{k}(x)\mu(x)\\
&=\frac{1}{2}\underset{x\in B_k}\sum\underset{y\sim x}\sum \omega_{xy}(u_{k}(y)-u_{k}(x))^2+\underset{x\in B_k}\sum h(x)u^2_{k}(x)\mu(x)+\frac{1}{2}\underset{x\in \partial B_k}\sum\underset{y\sim x}\sum \omega_{xy}(u_{k}(y)-u_{k}(x))^2\\
&\leq 2\|u_{k}\|^2_{\mathcal{H}(B_{k})}\\
&\leq 2C,
\end{align*}
which implies that $\{u_{k}\}$ is bounded in $\mathcal{H}$. Since $\mathcal{H}$ is weakly compact, it follows that up to a subsequence, $\{u_{k}\}$ converges to some function $u^{*}_{1}$ weakly in $\mathcal{H}$. This in particular implies
$$\int_{V}u_{k}\phi d\mu\rightarrow\int_{V}u^{*}_{1}\phi d\mu,\ \ \forall \phi\in C_{c}(V).$$
For any fixed point $z\in V$, taking
\begin{align*}
\phi(x)=\begin{cases}
1, &\text{if}\ \ x=z,\\
0, &\text{if}\ \ x\neq z.
\end{cases}
\end{align*}
Then $u_{k}(z)\rightarrow u^{*}_{1}(z).$
Hence by the uniqueness of the
limit, $u^{*}_{1}(z)=u^{*}(z)$ for all $z\in V$ and then $u^{*}\in \mathcal{H}$.
Thus, it follows from (\ref{locally1}) that for any fixed $x\in V$, there hold
$$-\Delta u^{*}+h(x)u^{*}=u^{*}\log (u^{*})^2.$$
Therefore $u^{*}$ is a solution of (\ref{log}).

Finally, we prove that $u^{*}$ is nontrivial.  Since $u_{k}$ is a nontrivial solution of equation \eqref{locally},
 $\langle \mathcal{I}_{k}'(u_{k}), u_{k}\rangle=0$.  By \eqref{inequality1} and Lemma~\ref{Sobolevembedding}, for any $0<\varepsilon<1$, there exists $C_{\varepsilon}$ such that
\begin{align*}
  \|u_{k}\|^{2}_{\mathcal{H}(B_k)} &= \int_{B_{k}}u_{k}^2\log u^{2}_{k}d\mu \\
   &\leq \int_{\{x\in B_{k}:u_{k}(x)\geq 1\}}u_{k}^2\log u^{2}_{k}d\mu\\
   &\leq C_{\varepsilon}\int_{B_{k}} |u_{k}|^{2+\varepsilon}d\mu \\
   &\leq C_{\varepsilon}C_{2}\|u_{k}\|^{2+\varepsilon}_{\mathcal{H}(B_{k})}.
\end{align*}
Thus, we obtain
\begin{equation*}
 \|u_{k}\|^2_{\mathcal{H}(B_{k})}\geq(C_{\varepsilon}C_{2})^{-\frac{2}{\varepsilon}}>0,
\end{equation*}
where $C_{\varepsilon}$ and $C_{2}$ is independent of $k$. Hence it implies that there exists a constant $\theta$ independent of $k$ satisfying $0<\theta<(C_{\varepsilon}C_{2})^{-\frac{2}{\varepsilon}}$ such that
\begin{equation}\label{M1}
 \int_{V}u_{k}^2\log u^{2}_{k}d\mu\geq\int_{B_{k}}u_{k}^2\log u^{2}_{k}d\mu=\|u_{k}\|^2_{\mathcal{H}(B_{k})}>\theta>0.
\end{equation}
Note that $u_k\rightharpoonup u^*$ in $\mathcal{H}$. Then it follows from \eqref{inequality1} that, for any $0<\varepsilon<1$, there exists $C_{\varepsilon}$ such that
\begin{equation*}
  \int_{V}|u_{k}^2\log u^{2}_{k}|d\mu
  \leq C_{\varepsilon}\int_{V}(|u_k|^{2-\varepsilon}+|u_k|^{2+\varepsilon})d\mu.
\end{equation*}
Thus by Lemma~\ref{Sobolevembedding}, Lebesgue dominated convergence theorem and \eqref{M1}, we have
\begin{equation*}
  \int_{V}(u^{*})^2\log (u^{*})^2d\mu=\underset{k\rightarrow\infty}\lim\int_{V}u_{k}^2\log u^{2}_{k}d\mu\geq \theta>0,
\end{equation*}
which implies that there must exist a vertex $x_0\in V$ such that $\log (u^*(x_0))^2>0$ and consequently there
holds $|u^*(x_0)|>1$. Therefore, $u^*\not\equiv 0$ and this completes the proof of the theorem.

\section{Existence of multiple solutions }
In this section, we will prove that the equation \eqref{log} admits infinitely many solutions by using the following symmetric mountain pass theorem.
\begin{lemma}[{\cite[Theorem 9.12]{Rabinowitz}}]\label{sym}
Let $X$ be an infinite dimensional Banach space, and let $J\in \mathcal{C}^{1}(X,\mathbb{R})$ be even, satisfying $(PS)$ condition and $J(0)=0$. If $X=Y\bigoplus Z$, where $Y$ is finite dimensional and $J$ satisfies

\noindent (i) there are constant $\varrho,\sigma>0$ such that $J|_{\partial B_{\varrho}\cap Z}\geq \sigma$ and

\noindent (ii) for each finite dimensional subspace $\widetilde{X}\subset X$ there exists an $R=R(\widetilde{X})$ such that $J\leq0$ on $\widetilde{X}\setminus B_{R(\widetilde{X})}$,

\noindent then $J$ possesses an unbounded sequence of critical values.
\end{lemma}

Now, we present a Sobolev embedding result, by which we know that the functional $\mathcal{J}$ defined by \eqref{functionalJ} is well defined in $\mathcal{W}$, where $\mathcal{W}$ is defined by \eqref{fW}.

\begin{lemma}[{\cite[Lemma 4]{ChangYan1}}]\label{e}
Let $G=(V,E)$ be a connected locally finite graph satisfying \eqref{mu} and \eqref{omega}. Assume that $(h'_{1})$ and $(h'_{2})$ hold.  Then $\mathcal{W}$ is weakly pre-compact and $\mathcal{W}$ is compactly embedded into $L^{q}(V)$ for any $q\in [1,+\infty]$.
\end{lemma}

The following lemma gives the same conclusion as Proposition 9 in \cite{ChangYan1},  which tells us that the functional $\mathcal{J}\in C^{1}(\mathcal{W},\mathbb{R})$, while the authors in \cite{ChangYan1} did not give a detailed proof of it. We supplement the proof here for the convenience of the readers.

\begin{lemma}\label{smooth}
Under the same assumptions as in Lemma~\ref{e},  the functional $\mathcal{J}\in C^{1}(\mathcal{W},\mathbb{R})$.
Furthermore, for any $v\in \mathcal{W}$,
\begin{equation}\label{Id}
  \langle \mathcal{J}^\prime(u),v\rangle=\int_{V}(\nabla u\nabla v+h(x)uv)d\mu-\int_{V}uv\log u^2d\mu.
\end{equation}
\end{lemma}

\begin{proof}
For any $u\in \mathcal{W}$, let
\begin{equation*}
  \varphi(u)=\frac{1}{2}\int_V(|\nabla u|^2+(h(x)+1)u^2)d\mu
\end{equation*}
and
\begin{equation*}
  \psi(u)=\frac{1}{2}\int_{V}u^2\log u^2d\mu.
\end{equation*}
Then $\mathcal{J}(u)=\varphi(u)-\psi(u)$.
We only need to prove that $\varphi, \psi\in C^{1}(\mathcal{W},\mathbb{R})$ and for any $v\in \mathcal{W}$,
\begin{equation*}
 \langle \varphi^\prime(u),v\rangle=\int_{V}(\nabla u\nabla v+(h(x)+1)uv)d\mu,
\end{equation*}
\begin{equation*}
  \langle \psi^\prime(u),v\rangle=\int_{V}(uv+uv\log u^2)d\mu.
\end{equation*}

First, we verify $\varphi\in C^{1}(\mathcal{W},\mathbb{R})$. Indeed, for any $v\in\mathcal{W}$,

\begin{align*}
(\varphi^\prime(u),v)&=\underset{t\rightarrow 0}\lim\frac{\varphi(u+tv)-\varphi(u)}{t}\\
&=\frac{1}{2}\underset{t\rightarrow 0}\lim\frac{1}{t}\left[\int_{V}(|\nabla (u+tv)|^{2}+( h(x)+1)(u+tv)^{2})d\mu-\int_{V}(|\nabla u|^{2}+(h(x)+1)u^{2})d\mu\right]\\
&=\int_{V}(\nabla u\nabla v+(h(x)+1)uv)d\mu.
\end{align*}
Let $u_{n}\rightarrow u$ in $\mathcal{W}$,  then $u_n\rightarrow u$ in $L^2(V)$. We claim that $\|\varphi^\prime(u_{n})-\varphi^\prime(u)\|_{\mathcal{W}^{*}}\rightarrow 0$ as $n\rightarrow \infty$, where $\mathcal{W}^{*}$ is the dual space of $\mathcal{W}$. In fact, for any $v\in \mathcal{W}$ with $\|v\|_{\mathcal{W}}=1$, we have
\begin{align*}
\langle\varphi'(u_{n})-\varphi'(u), v\rangle&=\int_{V}[\nabla(u_{n}-u)\nabla v+(h(x)+1)(u_{n}-u)v]d\mu\\
&=\int_{V}[\nabla(u_{n}-u)\nabla v+h(x)(u_{n}-u)v]d\mu+\int_{V}(u_{n}-u)vd\mu\\
&\leq\left(\int_{V}|\nabla (u_{n}-u)|^{2}d\mu\right)^{\frac{1}{2}}\left(\int_{V}|\nabla v|^{2}d\mu\right)^{\frac{1}{2}}\\
&\phantom{{}={}}+\left(\int_{V}h(x)(u_{n}-u)^{2}d\mu\right)^{\frac{1}{2}}
\left(\int_{V}h(x)v^{2}d\mu\right)^{\frac{1}{2}}
+\left(\int_{V}(u_{n}-u)^{2}d\mu\right)^{\frac{1}{2}}
\left(\int_{V}v^{2}d\mu\right)^{\frac{1}{2}}\\
&\leq 2\|u_{n}-u\|_{\mathcal{W}}\|v\|_{\mathcal{W}}+\|u_n-u\|_2\|v\|_2\rightarrow 0 \ \ \hbox{as}\ \ n\rightarrow\infty.
\end{align*}
Hence, $\|\varphi'(u_{n})-\varphi'(u)\|_{\mathcal{W}^{*}}=\underset{\|v\|_{\mathcal{W}}=1}\sup \langle\varphi'(u_{n})-\varphi'(u), v\rangle\rightarrow 0$ as $n\rightarrow\infty$. Therefore, $\varphi\in C^{1}(\mathcal{W},\mathbb{R})$.

Next, we prove that $\psi\in C^{1}(\mathcal{W},\mathbb{R})$.
Let $u,v\in \mathcal{W}\hookrightarrow L^{q}(V)$, $q\geq 1$, then we have $u,v\in L^{q}(V)$. Given $x\in V$ and $0<|t|<1$, by the mean value theorem and \eqref{inequality1}, there exists $\xi\in (0,1)$ such that
\begin{align*}
  &\frac{\left|(u+tv)^{2}\log (u+tv)^{2}-u^{2}\log u^{2}\right|}{|t|}\\
  =&2\left|(u+\xi tv)\cdot v\cdot\log (u+\xi tv)^2+(u+\xi tv)\cdot v\right|  \\
  \leq& 2C_\varepsilon(|u+\xi tv|^{1-\varepsilon}+|u+\xi tv|^{1+\varepsilon})|v|+2|u+\xi tv|\cdot|v|\\
  \leq& 2^{2-\varepsilon}C_\varepsilon(|u|^{1-\varepsilon}+|v|^{1-\varepsilon})|v|+ 2^{2+\varepsilon}C_\varepsilon(|u|^{1+\varepsilon}+|v|^{1+\varepsilon})|v|+2|u+v|\cdot|v|.
\end{align*}
The H\"{o}lder's inequality implies that
\begin{equation*}
\left[2^{2-\varepsilon}C_\varepsilon(|u|^{1-\varepsilon}+|v|^{1-\varepsilon})|v|+ 2^{2+\varepsilon}C_\varepsilon(|u|^{1+\varepsilon}+|v|^{1+\varepsilon})|v|+2|u+v|\cdot|v|\right]\in L^{1}(V).
\end{equation*}
It follows from  Lebesgue dominated convergence theorem that
\begin{align*}
\langle\psi'(u),v\rangle&=\underset{t\rightarrow 0}\lim\frac{\psi(u+tv)-\psi(u)}{t}\\
&=\frac{1}{2}\int_{V}\underset{t\rightarrow 0}\lim\frac{(u+tv)^{2}\log (u+tv)^{2}-u^{2}\log u^{2}}{t}d\mu\\
&=\int_{V}(uv+uv\log u^2)d\mu.
\end{align*}
Note that for any $0<\varepsilon<1$, there exists $C_{\varepsilon}>0$ such that
\begin{equation*}
  |u\log u^2|\leq C_{\varepsilon}(|u|^{1-\varepsilon}+|u|^{1+\varepsilon})
  =C_{\varepsilon}\left(|u|^{\frac{2-\varepsilon}{(2-\varepsilon)/(1-\varepsilon)}}
  +|u|^{\frac{2+\varepsilon}{(2+\varepsilon)/(1+\varepsilon)}}\right).
\end{equation*}
To prove the G\^{a}teaux derivative of $\psi$ is continuous,
we let $f(u)=u\log u^2$ and assume that $u_{n}\rightarrow u$ in $\mathcal{W}$ as $n\rightarrow\infty$.
Then by Lemma~\ref{e},  up to a subsequence, we have $u_{n}\rightarrow u$ in $L^{2-\varepsilon}(V)\cap L^{2+\varepsilon}(V)$, where $||\cdot||_{L^{2-\varepsilon}(V)
\cap L^{2+\varepsilon}(V)}:=||\cdot||_{{2-\varepsilon}}+||\cdot||_{2+\varepsilon}$.
It follows from Lemma~5.12 in \cite{HanShao} that 
\begin{equation*}
  f(u_{n})\rightarrow f(u)\ \ \hbox{in}\ \  L^{\frac{2-\varepsilon}{1-\varepsilon}}(V)+L^{\frac{2+\varepsilon}{1+\varepsilon}}(V).
\end{equation*}
Here, for any $w\in L^{r}(V)+ L^{s}(V), r,s\geq 1$, we define the norm
$$\|w\|_{L^{r}(V)+ L^{s}(V)}:=\inf\{||w_1||_{r}+||w_2||_{s}:w_1\in L^{r}(V),w_2\in L^{s}(V),w=w_1+w_2\}.$$
On the other hand, by the H\"{o}lder's inequality, we obtain
\begin{align*}
  |\langle\psi'(u_{n})- \psi'(u),v\rangle|&=\left|\int_{V}(u_n\log u^2_n-u\log u^2)vd\mu+\int_{V}(u_n-u)vd\mu\right|\\
  &\leq \|f(u_{n})-f(u)\|_{L^{\frac{2-\varepsilon}{1-\varepsilon}}(V)+L^{\frac{2+\varepsilon}{1+\varepsilon}}(V)}
  \|v\|_{L^{2-\varepsilon}(V)\cap L^{2+\varepsilon}(V)}+\|u_n-u\|_2\|v\|_2.
\end{align*}
Then
$$\|\psi'(u_{n})- \psi'(u)\|_{\mathcal{W}^{*}}\rightarrow 0 \ \ \hbox{as}\ \  n\rightarrow \infty,$$
which implies that $\psi\in C^{1}(\mathcal{W},\mathbb{R})$. Thus $\mathcal{J}\in C^{1}(\mathcal{W},\mathbb{R})$ and (\ref{Id}) holds.
\end{proof}

Next, we verify that the functional $\mathcal{J}$ satisfies the $(PS)$ condition.

\begin{lemma}\label{IPS}
Under the same assumptions as in Lemma~\ref{e}. Then for any $c\in \mathbb{R}$, $\mathcal{J}$ satisfies the $(PS)_{c}$ condition.
\end{lemma}

\begin{proof}
Assume that $\{u_n\}\subset \mathcal{W}$ satisfies $\mathcal{J}(u_n)\rightarrow c$ and $\mathcal{J}'(u_n)\rightarrow 0$ as $n\rightarrow \infty$. Then we obtain
\begin{equation}\label{Iu2}
  \|u_{n}\|_2=\int_{V}u_n^2d\mu=2\mathcal{J}(u_n)-\langle\mathcal{J}'(u_n),u_n\rangle\leq c+1+o(1)\|u_{n}\|_{\mathcal{W}}.
\end{equation}
Similar to the proof of \eqref{bdd2}, by Lemma~\ref{e}, H\"{o}lder's inequality and Yong inequality , we have for any $\varepsilon\in (0,1)$, there exist $C_{\varepsilon},C_{3}>0$ such that
\begin{align}\label{bdd5}
\int_{V}u^2_n\log u_n^2d\mu&\leq\int_{\{x\in V:u\geq1\}}u^2_n\log u_n^2d\mu\nonumber\\
&\leq C_{\varepsilon}\int_{V}|u_{n}|^{2+\varepsilon}d\mu\nonumber\\
&=C_{\varepsilon}\|u_{n}\|_{2}\|u_{n}\|^{1+\varepsilon}_{2(1+\varepsilon)}\\
&\leq C_{\varepsilon}C_{3}\|u_{n}\|_{2}\|u_{n}\|^{1+\varepsilon}_{\mathcal{W}}\nonumber\\
&\leq\frac{1+\varepsilon}{2}\|u_{n}\|^{2}_{\mathcal{W}}+
\frac{1-\varepsilon}{2}C'_{\varepsilon}C'_{3}\|u_{n}\|^{\frac{2}{1-\varepsilon}}_{2}.\nonumber
\end{align}
Note that $\langle\mathcal{J}'(u_{n}),u_{n}\rangle=0$, then by \eqref{Iu2} and \eqref{bdd5}, we obtain
\begin{align*}
   \|u_{n}\|^2_{\mathcal{W}} &= \int_{V}u^2_n\log u_n^2d\mu+\|u_{n}\|^2_2\\
  &\leq \frac{1+\varepsilon}{2}\|u_{n}\|^{2}_{\mathcal{W}}+
\frac{1-\varepsilon}{2}C'_{\varepsilon}C'_{3}\|u_{n}\|^{\frac{2}{1-\varepsilon}}_{2}+\|u_{n}\|^2_2,
\end{align*}
which implies that $\{u_n\}$ is bounded in $\mathcal{W}$. Since $\mathcal{W}$ is weakly pre-compact, there exists $u\in \mathcal{W}$ such that, up to a subsequence,
\begin{equation}\label{weak}
  \left\{\aligned&u_{n} \rightharpoonup u  &\hbox{in} \ \ \mathcal{W},\\
&u_{n}(x)\rightarrow u(x) &\forall x\in V,\\
&u_{n} \rightarrow u &\hbox{in}\ \ L^{q}(V),\ \ q\in [1,+\infty].
\endaligned\right.
\end{equation}
By direct calculation, one gets
\begin{align}\label{weaka1}
  \|u_{n}-u\|^2_{\mathcal{W}} &= \langle \mathcal{J}'(u_{n}),u_n-u\rangle-\langle \mathcal{J}'(u),u_n-u\rangle+\|u_n-u\|^2_2\nonumber \\
   &+\int_{V}u_n(u_n-u)\log u^2_nd\mu-\int_{V}u(u_n-u)\log u^2d\mu
\end{align}
It is easy to prove that
\begin{equation}\label{weak1}
  \langle \mathcal{J}'(u_{n}),u_n-u\rangle-\langle \mathcal{J}'(u),u_n-u\rangle\rightarrow0\ \
\hbox{and} \ \ \|u_n-u\|^2_2\rightarrow0\ \ \hbox{as}\ \ n\rightarrow\infty
\end{equation}
by using \eqref{weak}. It follows from \eqref{inequality1}, \eqref{weak} and the H\"{o}lder's inequality that
\begin{align}\label{weaka2}
  \left|\int_{V}u_n(u_n-u)\log u^2_nd\mu\right|&\leq\int_{V}|u_n(u_n-u)\log u^2_n|d\mu\nonumber\\
  &\leq\int_{V}C_{\varepsilon}|u_{n}-u|(|u_n|^{1-\varepsilon}+|u_n|^{1+\varepsilon})d\mu\\
   &\leq C_{\varepsilon}(\|u_{n}-u\|_{2}\|u_n\|^{1-\varepsilon}_{2(1-\varepsilon)}
   +\|u_{n}-u\|_{2}\|u_n\|^{1+\varepsilon}_{2(1+\varepsilon)})\nonumber\\
   &\rightarrow 0 \ \ \hbox{as}\ \ n\rightarrow\infty. \nonumber
\end{align}
Similarly, we have
\begin{equation}\label{weak2}
\int_{V}u(u_n-u)\log u^2d\mu\rightarrow0\ \ \hbox{as}\ \ n\rightarrow\infty.
\end{equation}
Thus, combining \eqref{weaka1},\eqref{weak1}, \eqref{weaka2} and \eqref{weak2}, we obtain
\begin{equation*}
  \|u_{n}-u\|^2_{\mathcal{W}}\rightarrow0 \ \ \hbox{as}\ \ n\rightarrow \infty.
\end{equation*}
Then we complete the proof.
\end{proof}

\vskip4pt

\emph{The Proof of Theorem~\ref{existence2}.}~~Let $X=\mathcal{W}$ in Lemma~\ref{sym}. Next, we prove that the functional $\mathcal{J}$ satisfies conditions $(i)$ and $(ii)$ in Lemma~\ref{sym}.

\noindent $(i)$ By \eqref{inequality1} and Lemma~\ref{e}, for any $0<\varepsilon<1$, there  exists $C_{\varepsilon}>0$ such that
\begin{align*}
  \mathcal{J}(u) &= \frac{1}{2}\|u\|^2_{\mathcal{W}}-\frac{1}{2}\int_{V}u^2\log u^2d\mu \\
   &\geq\frac{1}{2}\|u\|^2_{\mathcal{W}}-\frac{1}{2}\int_{\{x\in V:u\geq 1\}}u^2\log u^2d\mu\\
   &\geq\frac{1}{2}\|u\|^2_{\mathcal{W}}-\frac{1}{2}C_{\varepsilon}\|u\|^{2+\varepsilon}_{2+\varepsilon}\\
   &\geq\frac{1}{2}\|u\|^2_{\mathcal{W}}-\frac{1}{2}C_{\varepsilon}C_{4}\|u\|^{2+\varepsilon}_{\mathcal{W}}.
\end{align*}
Taking $\varrho=\left(\frac{1}{2C_{\varepsilon}C_{4}}\right)^{\frac{1}{\varepsilon}}$, then we have
\begin{equation*}
  \mathcal{J}(u)|_{\partial B_{\varrho}} \geq \frac{1}{4}\varrho^2:=\sigma>0.
\end{equation*}

\noindent $(ii)$ Let $\widetilde{\mathcal{W}}$ be a  finite dimensional subspace of $\mathcal{W}$. Suppose $u\in \widetilde{\mathcal{W}}\setminus\{0\}$ and setting $\phi:=\frac{u}{\|u\|_{\mathcal{W}}}$. Then we have
\begin{align*}
  \mathcal{J}(u) &=\frac{1}{2}\|u\|^2_{\mathcal{W}}-\frac{1}{2}\int_{V}u^2\log u^2d\mu\\
                 &=\frac{\|u\|^2_{\mathcal{W}}}{2}\left(\|\phi\|^2_{\mathcal{W}}
                 -\int_{V}\phi^2\log\phi^2d\mu-\int_{V}\phi^2\log\|u\|^2_{\mathcal{W}}\right)\\
                 &\leq\frac{\|u\|^2_{\mathcal{W}}}{2}\left(1+C_{\varepsilon}\|\phi\|^{2-\varepsilon}_{2-\varepsilon}
  +C_{\varepsilon}\|\phi\|^{2+\varepsilon}_{2+\varepsilon}-\|\phi\|^2_2\log\|u\|^2_{\mathcal{W}}\right).
\end{align*}
Therefore, there exists an $R=R(\widetilde{\mathcal{W}})$ such that $\mathcal{J}\leq 0$ on $\widetilde{\mathcal{W}}\setminus B_{R(\widetilde{\mathcal{W}})}$.


From the above argument, we know that the functional $\mathcal{J}$ satisfies all the conditions of Lemma~\ref{sym}. Therefore, the equation \eqref{log} has a sequence of solutions $\{u_n\}\subset \mathcal{W}$ such that $\mathcal{J}(u_n)\rightarrow +\infty$ as $n\rightarrow\infty$.

\section*{Acknowledgements}
This research is supported by National Natural Science Foundation of China (No. 12101355) and the Open Project Program (K202303) of Key Laboratory of Mathematics and Complex System, Beijing Normal University. Part of this work was carried out while this author was visiting Tsinghua University.

\end{document}